\documentclass[oneside]{amsart}
\usepackage{etex}
\usepackage[a4paper]{geometry}
\usepackage{amssymb,amsfonts,amsmath}
\usepackage{comment} 
\usepackage[all,arc]{xy}
\usepackage{enumerate}
\usepackage{mathrsfs,mathtools}
\usepackage{todonotes,booktabs}
\usepackage{stmaryrd}
\usepackage{marvosym}

\usepackage{bbm}
\usepackage{tikz}
\usepackage{tikz-cd}
\usetikzlibrary{trees}
\usetikzlibrary[shapes]
\usetikzlibrary[arrows]
\usetikzlibrary{patterns}
\usetikzlibrary{fadings}
\usetikzlibrary{backgrounds}
\usetikzlibrary{decorations.pathreplacing}
\usetikzlibrary{decorations.pathmorphing}
\usetikzlibrary{positioning}
	
\def\biblio{\bibliography{duality}\bibliographystyle{alpha}}

\usepackage{xcolor} 
\usepackage{graphicx}

\usepackage[pagebackref]{hyperref}

\definecolor{dark-red}{rgb}{0.5,0.15,0.15}
\definecolor{dark-blue}{rgb}{0.15,0.15,0.6}
\definecolor{dark-green}{rgb}{0.15,0.6,0.15}
\hypersetup{
    colorlinks, linkcolor=dark-red,
    citecolor=dark-blue, urlcolor=dark-green
}

\newcommand{\lra}[1]{\overset{#1}{\longrightarrow}}

\renewcommand*{\backref}[1]{}
\renewcommand*{\backrefalt}[4]{%
  \ifcase #1 %
No citations.
  \or
(cit. on p. #2).%
  \else
(cit on pp. #2).%
  \fi%
}
\usepackage{subfiles}

\usepackage[nameinlink,capitalise,noabbrev]{cleveref}

\newtheorem{thm}{Theorem}[section]
\newtheorem{cor}[thm]{Corollary}
\newtheorem{prop}[thm]{Proposition}
\newtheorem{lem}[thm]{Lemma}

\newtheorem{quest}[thm]{Question}

\theoremstyle{definition}
\newtheorem{defn}[thm]{Definition}

\newtheorem{ex}[thm]{Example}

\theoremstyle{remark}
\newtheorem{rem}[thm]{Remark}

\theoremstyle{theorem}
\newtheorem*{thm*}{Theorem}

\newtheorem*{prop*}{Proposition}

\makeatletter
\let\c@equation\c@thm
\makeatother
\numberwithin{equation}{section}

\DeclareMathOperator{\colim}{colim}

\DeclareMathOperator{\Spec}{Spec}
\DeclareMathOperator{\Spf}{Spf}

\newcommand{\Q}{\mathbb{Q}}

\newcommand{\bE}{\mathbb{E}}

\DeclareMathOperator{\Level}{Level}
\newcommand{\Prod}[1]{\underset{#1}{\prod}}

\newcommand{\E}{E_{n}}

\newcommand{\Z}{\mathbb{Z}}

\Crefname{figure}{Figure}{Figures}
\Crefname{assu}{Assumption}{Assumptions}
\Crefname{lem}{Lemma}{Lemmas}

\newcommand{\lkt}{\Lambda_{k,t}}

\newcommand{\fp}{\mathfrak{p}}
\newcommand{\fm}{\mathfrak{m}}
\newcommand{\fq}{\mathfrak{q}}

\newcommand{\recollement}[5]{
\xymatrix{{#1} \ar[r]|-{#2} & #3 \ar[r]|-{#4} \ar@<1ex>[l]^-{{#2}_!} \ar@<-1ex>[l]_-{{#2}^*} & #5, \ar@<1ex>[l]^-{{#4}!} \ar@<-1ex>[l]_-{{#4}^*}
}}
\let\lim\relax

\DeclareMathOperator{\lim}{lim}

\newcommand{\bG}{\mathbb{G}}
\newcommand{\bW}{\mathbb{W}}

\newcommand{\cL}{\mathcal{L}}

\title{Excellent rings in transchromatic homotopy theory}

\author{Tobias Barthel}
\address{Department of Mathematical Sciences, University of Copenhagen, Universitetsparken 5, 2100 K{\o}benhavn {\O}, Denmark}
\email{tbarthel@math.ku.dk}

\author{Nathaniel Stapleton}

\address{Fakult{\"a}t f{\"u}r Mathematik, Universit{\"a}t Regensburg, 93040 Regensburg, Germany}
\email{nat.j.stapleton@gmail.com}

\date{\today}

\begin{document}


\begin{abstract}
The purpose of this note is to verify that several basic rings appearing in transchromatic homotopy theory are Noetherian excellent normal domains and thus amenable to standard techniques from commutative algebra. In particular, we show that the coefficients of iterated localizations of Morava $E$-theory at the Morava $K$-theories are normal domains and also that the coefficients in the transchromatic character map for a fixed group form a normal domain.
\end{abstract}

 
\maketitle

\def\biblio{}

\section{Introduction}

Excellent rings were introduced by Grothendieck as a well-behaved class of commutative Noetherian rings general enough for the purposes of arithmetic and algebraic geometry, 
while excluding several pathological examples of Noetherian rings found by Nagata. In particular, the collection of excellent rings is closed under localization and completion. 

These algebraic operations describe the effect on coefficient rings of the derived localizations appearing in stable homotopy theory.  Most prominently, such Bousfield localizations occur when comparing different chromatic layers of the stable homotopy category, a subject known as transchromatic homotopy theory. The main goal of this note is to demonstrate that important rings appearing in transchromatic homotopy theory are built from excellent rings, and hence surprisingly well-behaved. Specifically, although these rings are rather complicated algebraically and in general not regular, we prove that they are integral domains and thus regular in codimension 1. However, establishing these fundamental properties directly turned out to be considerably more difficult than anticipated, which led us to employ the theory of excellent rings instead. We hope that the methods used here will prove useful in tackling similar problems in related contexts. 

Our first result concerns the rings $L_{t,n}=\pi_0L_{K(t)}E_n$ obtained as the localization of Morava $E$-theory $E_n$ of height $n$ at the Morava $K$-theory $K(t)$ of height $t<n$. This is a fundamental example of a transchromatic ring, since the localization map
\[
\xymatrix{E_n \ar[r] & L_{K(t)}E_n}
\]
shifts chromatic height from $n$ to $t$. These rings were studied in detail by Mazel--Gee, Peterson, and Stapleton in \cite{piperings} using the theory of pipe rings;  in particular, they show that $L_{t,n}$ represents a natural moduli space. We complement their work by showing that the $L_{t,n}$ are well-behaved from the point of view of commutative algebra. 

\begin{thm*}[\Cref{prop:lt}]
The ring $L_{t,n}$ is a Noetherian excellent normal domain. More generally, the same conclusion holds for any iteration of localizations of $E_n$ at Morava $K$-theories. 
\end{thm*}

As a consequence, $L_{t,n}$ has the cancellation property, which is crucial in applications as for example in \cite{bs_bpetheory}. Moreover, this result about $L_{t,n}$ forms the basis for deducing similar properties of other prominent rings appearing in transchromatic homotopy theory. Specifically, we study the completion at $I_t$ of the transchromatic character rings $C_{t,k}$ 
\[
\hat{C}_{t,k} = (C_{t,k})^{\wedge}_{I_t}.
\]
that were first introduced in \cite{stapleton_tgcm}. For a fixed finite group $G$ there exists a $k \geq 0$ so that $\hat{C}_{t,k}$ may be used to build a transchromatic character map
\[
\hat{C}_{t,k} \otimes_{E_{n}^0} E_{n}^0(BG) \lra{\cong} \hat{C}_{t,k} \otimes_{L_{t,n}} L_{K(t)}E_{n}^{0}(BG^{B\Z_{p}^{n-t}}).
\]

\begin{thm*}[\Cref{cor:ctk} and \Cref{cor:ct}]
The transchromatic character ring $\hat{C}_{t,k}$ is a Noetherian excellent normal domain for all $t$ and $k$. The colimit $\colim_k \hat{C}_{t,k}$ is normal.
\end{thm*}

There are two key ingredients in the proof of this theorem: Firstly, a recent theorem of Gabber--Kurano--Shimomoto on the ideal-adic completion of excellent rings. Secondly, we use an identification of $\hat{C}_{t,k}$ with a localization and completion of a certain ring of Drinefeld level structure on the formal group associated to $E_n$, thereby providing a new perspective on these transchromatic character rings. 

\subsection*{Conventions and references}

We rely heavily on a number of results from commutative algebra that cannot be found in standard textbooks. Rather than locating the earliest published reference for each of the facts used here, we will always refer to the stacks project \cite{stacks-project}. All rings in this note are assumed to be commutative and all ideals are taken to be finitely generated. 

\subsection*{Acknowledgements}

We would like to thank Frank Gounelas for providing helpful geometric intuition and Niko Naumann and Paul VanKoughnett for useful discussions. The second author is supported by SFB 1085 \emph{Higher Invariants} funded by the DFG.

\section{Excellent rings}

We start this section with an example illustrating that localizations of complete regular local Noetherian rings are more complicated than what one might expect. This shows that, while the localization or completion of a regular local ring is again regular, the class of regular local rings is not closed under these operations. 

\begin{ex}
Let $k$ be a field and consider $R = k\llbracket x,y \rrbracket$. The ring $A = y^{-1}R \cong k\llbracket x \rrbracket (\!(y)\!)$ is not local and, in particular, it is not isomorphic to the regular local ring $k(\!(y)\!)\llbracket x \rrbracket$. Indeed, we claim that both $(x)$ and $(x-y)$ are different maximal ideals in $A$. The ideal $(x)$ is clearly maximal, so it remains to show that $(x-y)$ is maximal as well. To this end, note that 
\[
A/(x-y) \cong k(\!(x)\!),
\]
which is a field, hence $(x-y) \subset A$ is maximal as claimed. In contrast, $(x-y) = x(1-y/x)$ has formal inverse $1/x \cdot \sum_{i=0}^{\infty}(y/x)^i$, so $x-y$ is a unit in $k(\!(y)\!)\llbracket x \rrbracket$.

After completion, this subtlety disappears: in fact, $A_{(x)}^{\wedge} \cong k(\!(y)\!)\llbracket x \rrbracket$, see \cite[B.2]{piperings}.
\end{ex}

The Auslander--Buchsbaum theorem asserts that regular local rings are unique factorization domains, but in light of the previous example, we may not expect this property to be preserved under the operations appearing in transchromatic theory. Instead, we will study the larger class of normal rings, which corresponds to regularity in codimension 1 via Serre's criterion. 

To this end, recall that a normal domain is a domain which is integrally closed in its quotient field. A ring $R$ is called normal if the localizations $R_{\fp}$ are normal domains for all primes ideals $\fp \subset R$. The following lemma collects the key properties of normal rings that we will apply in this note.

\begin{lem}\label{lem:normal}
Let $R$ be a commutative ring.
\begin{enumerate}
	\item If $R$ is regular, then $R$ is normal.
	\item If $R$ is normal, then any of its localizations is normal.
	\item Filtered colimits of normal rings are normal.
	\item If $R$ is Noetherian and normal, then it is a finite product of normal domains.
\end{enumerate}
\end{lem}
\begin{proof}
The first claim is \cite[Tag 0567]{stacks-project}, the second one is \cite[Tag 037C]{stacks-project}, and the third one is \cite[Tag 037D]{stacks-project}. To prove the last claim we apply \cite[Tag 030C]{stacks-project}, so that we have to check that $R$ is reduced and contains only finitely many minimal primes. Because $R$ is Noetherian, the second condition is satisfied. For the first one, note that, since $R$ is normal it is reduced, as being reduced is a local property and $R$ is locally a domain. 
\end{proof}

The definition of an excellent ring is rather involved. For the convenience of the reader, we give this definition; we will assume that the reader is familiar with the definition of a regular local ring. 

\begin{defn}
Recall the following definitions for a commutative ring $R$:
	\begin{enumerate}
		\item A ring $R$ is a $G$-ring if for all prime ideals $\fp \subset R$, $R_{\fp}$ is a local $G$-ring. A local ring $(R,\fm)$ is a local $G$-ring if the completion map $R \rightarrow R^{\wedge}_{\fm}$ is a regular morphism. This means that the map is flat and for all primes $\fp \subset R$, $\kappa(\fp) \otimes_R R^{\wedge}_{\fm}$ is Noetherian and geometrically regular over $\kappa(\fp)$. A $k$-algebra $R$ is geometrically regular over $k$ if for every finite extension $k \subset K$, $R \otimes_k K$ is regular, where a ring is regular if it is locally regular.
		\item A ring $R$ is J-2 if for all finite type extensions $R \rightarrow S$ the ring $S$ is J-1. A ring $R$ is J-1 if the subset of points of $\fp \in \Spec(R)$ with the property that $R_\fp$ is regular local is open.
		\item A ring $R$ is universally catenary if it is Noetherian and for every finite type extension $R \rightarrow S$, the ring $S$ is catenary. A ring $R$ is catenary if for all pairs of primes ideals $\fq \subset \fp \subset R$, all maximal chains of prime ideals $\fq = P_0 \subset P_1 \subset \ldots \subset P_l = \fp$ have the same length.
	\end{enumerate}
Finally, a ring $R$ is called excellent if it is Noetherian, a $G$-ring, J-2, and universally catenary. 
\end{defn}

For example, fields, Dedekind domains with characteristic $0$ quotient field, and all complete local Noetherian rings are excellent, see \cite[Tag 07QW]{stacks-project}. Moreover, the same reference shows that any algebra of finite type over an excellent ring is excellent.

The next proposition essentially generalizes \cite[Tag 0C23]{stacks-project} to non-local rings. The purpose is to show that the collection of Noetherian excellent normal rings is closed under the operations of localization at a multiplicatively closed set and completion at a prime ideal. The key point is that we do not assume that our rings are local, as the rings that naturally arise in transchromatic homotopy theory are often not local.


\begin{prop}\label{prop:normal}
Suppose $R$ is a Noetherian excellent normal ring, $\fp \subset R$ is a prime ideal, and $S \subseteq R \setminus \fp$ is multiplicatively closed, then $A = (R[S^{-1}])_{\fp}^{\wedge}$ is a Noetherian excellent normal domain. 
\end{prop}
\begin{proof}
By \cite[Tag 07QU]{stacks-project} and Part (2) of \Cref{lem:normal}, $R[S^{-1}]$ is Noetherian, excellent, and normal and $\fp R[S^{-1}]$ is prime, thus we may assume without loss of generality that $S = \{1\}$. Consider the canonical map $f\colon R \to A=R_{\fp}^{\wedge}$. Since $R$ is excellent and thus a $G$-ring, the map $f$ is regular by \cite[Tag 0AH2]{stacks-project}. The completion of a Noetherian ring is Noetherian, so we may apply \cite[Tag 0C22]{stacks-project} and Part (1) of \Cref{lem:normal} to deduce that $A$ is normal as desired. Since $A$ is Noetherian and normal, Part (4) of \Cref{lem:normal} implies that it is a product of finitely many domains $A_1, \ldots, A_l$. Also, $\fp A$ is prime. Now we have canonical isomorphisms
\[
A \cong A^{\wedge}_{\fp} \cong \underset{i}{\prod} (A_i)^{\wedge}_{\fp}.
\]
By the structure theory of prime ideals in products, only one of the factors will survive, hence $A \cong A_i$ for some $i$. Finally, the completion of an excellent ring is excellent by a recent theorem of Gabber--Kurano--Shimomoto \cite{completionexcellent}. 
\end{proof}

\begin{rem}
In virtue of \cite[Tag 07PW]{stacks-project}, the proof of the above proposition does not obviously work if we assume that $R$ is a Noetherian normal $G$-ring. We see that the conditions on $R$ all conspire to make the proof go through.
\end{rem}

\begin{prop} \label{primeideal}
Let $R$ be a domain, let $\fp = (r_1, \ldots, r_l) \subset R$ be a prime ideal generated by a regular sequence, and let $\fq = (r_1, \ldots, r_m)$ for $m<l$, then $\fq$ is a prime ideal. 
\end{prop}
\begin{proof}
Since $R$ is a domain, the localization map $R \rightarrow R_{\fp}$ is injective. The ring $R_{\fp}$ is regular local with system of parameters given by $r_1, \ldots, r_l$. In such a situation $R_{\fp}/(r_1, \ldots, r_m)$ is prime for any $1 \leq m \leq l$. Now the preimage of $(r_1, \ldots, r_m) \subset R_{\fp}$ is the ideal generated by $(r_1, \ldots, r_m) \subset R$ since the localization map $R \rightarrow R_{\fp}$ is injective.
\end{proof}

\section{Rings in transchromatic homotopy theory}

Throughout this section, we fix a prime $p$ and height $n\ge 0$. Recall that Morava $E$-theory $E_n$ is an even periodic $\bE_{\infty}$-ring spectrum with coefficients
\[
E_n^* := \pi_{-*}E_n \cong \bW k \llbracket u_1,\ldots,u_{n-1} \rrbracket[u^{\pm 1}],
\]
where $\bW k$ is the ring of Witt vectors on a perfect field $k$ of characteristic $p$, the $u_i$'s are in degree 0 and $u$ has degree $2$. Note that $\pi_0E_n = E_n^0$ is a complete regular local Noetherian ring, so in particular an excellent domain. Furthermore, let $K(n)$ be Morava $K$-theory of height $n$ with coefficients $K(n)_* \cong k[u^{\pm 1}]$ and denote by $L_{K(n)}$ the corresponding Bousfield localization functor. If $m<n$, then $L_{K(n)}L_{K(m)}=0$, but the composite $L_{K(m)}L_{K(n)}$ is non-trivial and encodes much of the structure of transchromatic homotopy theory, see for example \cite{hovey_csc}.  

Since all spectra involved are even periodic, we will restrict attention to the degree 0 part of the homotopy groups. In particular, an even periodic module $M$ over an even periodic $\bE_1$-ring spectrum $A$ is said to be flat if $\pi_0M$ is flat as $\pi_0A$-module. This definition is compatible with the one given in \cite{frankland}. 

\begin{lem}\label{lem:flatlt}
Given a sequence $0 \le t_1 \le \cdots \le t_i \le n$ of integers, the canonical localization map 
\[
\xymatrix{M \ar[r] & L_{K(t_1)}\cdots L_{K(t_i)}M }
\]
is flat for any flat $E_n$-module $M$. 
\end{lem} 
\begin{proof}
We will prove this by induction on the number $i$ of integers in the sequence. If $i=0$, the claim is trivial, so suppose it is proven for all sequences of numbers $t_2 \le \cdots \le t_i$ and let $t_1 \le t_2$; for simplicity, write $N= L_{K(t_2)}\cdots L_{K(t_i)}M$. By assumption, $\pi_0N$ is a flat $E_n^0$-module, so \cite[Cor.~3.10]{bs_centralizers} shows that 
\[
\pi_0L_{K(t_1)}N \cong (\pi_0N[u_{t_1}^{-1}])_{I_{t_1}}^{\wedge},
\]
where $I_t$ denotes the ideal $(p,u_1,\ldots,u_{t-1})$. Localization is exact, hence $\pi_0N[u_{t_1}^{-1}]$ is flat over $E_n^0$. Therefore, an unpublished theorem of Hovey, proven in \cite[Prop.~A.15]{frankland}, implies that $\pi_0L_{K(t_1)}N$ is flat over $E_n^0$ as claimed. 
\end{proof}

To simplify notation, we shall write $L_{K(T)}$ for the composite functor $L_{K(t_1)}\cdots L_{K(t_i)}$ for any sequence $T = (t_1,\ldots,t_i)$ of integers. Note that, if there is $j$ with $t_{j-1}>t_{j}$ in $T$, then $L_{K(T)} \simeq 0$.

\begin{prop}\label{prop:lt}
The ring $\pi_0L_{K(T)}E_n$ is a Noetherian excellent normal domain for all finite non-increasing sequences $T$ of non-negative integers. 
\end{prop}
\begin{proof}
We proceed by induction, the base case being clear. Suppose that the claim has been proven for all sequences of length at most $i-1$ and consider a sequence $T = \{t_1 \cup T'\}$ with $T' = (t_2, \ldots, t_i)$ of length $i-1$. Write $R = L_{K(T')}E_n$, so $\pi_0R$ is a Noetherian excellent normal domain by hypothesis. As is shown in the proof of \Cref{lem:flatlt}, there is an isomorphism $\pi_0L_{K(t_1)}R \cong (\pi_0R[u_{t_1}^{-1}])_{I_{t_1}}^{\wedge}$. The ideal $I_{t_1} \subset I_{t_{2}}$ satisfies the condition of \cref{primeideal} and therefore is prime. Thus the triple $(\pi_0R,\{u_{t_1},u_{t_1}^2,\ldots\}, I_{t_1})$ satisfies the assumptions of \Cref{prop:normal}, hence $\pi_0L_{K(t_1)}R$ is a Noetherian excellent normal domain.
\end{proof}

As a special case of \Cref{prop:lt}, we immediately obtain the following corollary, which was used in the proof of \cite[Lem.~4.4]{bs_bpetheory} and provided the original motivation for this note:

\begin{cor}\label{cor:lt}
The ring $L_t = \pi_0L_{K(t)}E_n$ is an excellent domain for all $t$ and $n$. In particular, $L_t$ has the cancellation property. 
\end{cor}

In \cite{bs_bpetheory}, we also studied a variant $F_t = L_{K(t)}((E_n)_{I_t})$ of $L_t$, where the localization $(-)_{I_t}$ is understood in the ring-theoretic sense, i.e., as inverting the complement of $I_t$. 

\begin{cor}\label{cor:ft}
The ring $\pi_0F_t$ is an excellent domain. 
\end{cor}
\begin{proof}
We see as in the proof of \Cref{prop:lt} that the (degree 0) coefficients of $F_t$ are given by $((E_n^0)_{I_t})_{I_t}^{\wedge}$, to which we can apply \Cref{prop:normal}. 
\end{proof}

We now turn to the rings that feature prominently in the transchromatic character theory of Hopkins, Kuhn, and Ravenel \cite{hkr}, as well as its generalizations~\cite{stapleton_tgcm} and~\cite{bs_centralizers}. To this end, we quickly review the definition and role of the coefficient ring for transchromatic character theory. For any integer $0 \le t \le n$, let $\bG_{L_{K(t)}E_n}$ be the formal group associated to the natural map $MU \to E_n \to L_{K(t)}E_n$, viewed as a $p$-divisible group. In \cite{stapleton_tgcm}, an $L_t$-algebra called $C_t$ is defined which carries the universal isomorphism of $p$-divisible groups
\[
C_t \otimes \bG_{\E} \cong (C_t \otimes \bG_{L_{K(t)}E_n}) \oplus (\Q_p/\Z_p)^{n-t}. 
\]
Note that both $L_t$ and $C_t$ also depend on $n$. From the perspective of stable homotopy theory, $C_t$ is useful because there is a canonical isomorphism (the character map)
\[
C_t \otimes_{E_{n}^0} E_{n}^0(BG) \lra{\cong} C_t \otimes_{\pi_0 L_{K(t)}E_n} L_{K(t)}E_{n}^0(\cL^{n-t}BG),
\]
where $\cL$ denotes the ($p$-adic) free loop space. The ring $C_t$ is a colimit of smaller rings $C_t = \colim_k C_{t,k}$. With $p$ and $n$ fixed implicitly, denote the finite abelian group $(\Z/p^k)^{n-t}$ by $\Lambda_{k,t}$. The ring $C_{t,k}$ is a localization of $L_t \otimes_{\E^0} \E^0(B\lkt^*)$. The ring $\E^0(B\lkt^*)$ carries the universal homomorphism 
\[
\lkt \rightarrow \bG_{E_n}.
\]

Recall from \cite[Section 2]{stapleton_tgcm} that the $p$-divisible group $L_t \otimes_{\E^0} \bG_{E_n}$ is the middle term in a short exact sequence
\[
0 \rightarrow \bG_{L_t} \rightarrow L_t \otimes_{\E^0} \bG_{E_n} \rightarrow \bG_{et} \rightarrow 0,
\]
where $\bG_{et}$ is a height $n-t$ {\'e}tale $p$-divisible group.

Let $T_{t,k} \subset \E^0(B\lkt^*)$ be the multiplicative subset generated by the nonzero image of the canonical map
\[
\lkt \rightarrow \bG_{E_n}(\E^0(B\lkt^*)).
\]
The nonzero image of this map has an explicit description in terms of a coordinate, as we shall explain now. After fixing an isomorphism $\mathcal{O}_{\bG_{E_n}} \cong E_{n}^0\llbracket x \rrbracket$, there is an induced isomorphism
\[
\E^0(B\lkt^*) \cong \E^0\llbracket x_1, \ldots, x_t \rrbracket /([p^k](x_1), \ldots, [p^k](x_t)),
\]
where $[p^k](x)$ is the $p^k$-series of the formal group law determined by the coordinate. The nonzero image of $\lkt$ can be described as the set of nonzero sums
\[
[i_1](x_1) +_{\bG_{E_n}} \ldots +_{\bG_{E_n}} [i_t](x_t).
\]
Of course, we may view $T_{k,t}$ as a subset of $L_{t} \otimes_{\E^0} \E^0(B\lkt^*)$.

Let $S_{t,k} \subset L_{t} \otimes_{\E^0} \E^0(B\lkt^*)$ be the multiplicative subset generated by the nonzero image of the canonical map
\[
\lkt \rightarrow (L_t \otimes_{\E^0} \bG_{E_n})(L_{t} \otimes_{\E^0} \E^0(B\lkt^*)) \rightarrow \bG_{et}(L_{t} \otimes_{\E^0} \E^0(B\lkt^*)).
\]

The ring $C_{t,k}$ is defined to be $S_{t,k}^{-1}(L_{t} \otimes_{\E^0} \E^0(B\lkt^*))$. Instead of working with this ring, we will work with the mild variation obtained by completing at $I_t$
\[
\hat{C}_{t,k} = (C_{t,k})^{\wedge}_{I_t} = (S_{t,k}^{-1}(L_{t} \otimes_{\E^0} \E^0(B\lkt^*)))^{\wedge}_{I_t}.
\]

For a finite abelian group $A$, the scheme of $A$-level structures in the formal group $\bG_{E_n}$ of Morava $E$-theory is represented by a ring $D_A$:
\[
\Level(A,\bG_{E_n}) \cong \Spf(D_A).
\]
This ring was introduced by Drinfeld \cite{drinfeld_ell1}. It was also studied further by Strickland in \cite{strickland_finitesubgroups}.  

\begin{lem}
The ring $D_{\Lambda_{k,t}}$ is a Noetherian excellent normal domain.
\end{lem}
\begin{proof}
The ring $D_{\Lambda_{k,t}}$ is a module-finite extension of $\E^{0}$. This immediately implies that it is Noetherian and excellent. Drinfeld proves that it is regular local and this implies that it is a normal domain.
\end{proof}

Let $A^*$ be the Pontryagin dual of $A$. The rings $\E^0(B\lkt^*)$ and $D_{\lkt}$ are closely related. There is a canonical surjective map
\[
\pi \colon \E^0(B\lkt^*) \twoheadrightarrow D_{\lkt}
\]
and the kernel is understood by the proof of Proposition 4.3 in \cite{drinfeld_ell1}. It is generated by power series $f_i(x_1, \ldots, x_i)$ for $1 \leq i \leq n-t$, where
\[
f_i(x_1, \ldots, x_i) = \frac{[p^k](x_i)}{g_i(x_1, \ldots, x_i)}
\]
and 
\[
g_i(x_1, \ldots, x_i) = \Prod{(j_1, \ldots, j_{i-1}) \in \Lambda_{k,i-1}} \big ( x_i - [j_1](x_1) +_{\bG_{E_n}} \ldots +_{\bG_{E_n}} [j_{i-1}](x_{i-1}) \big ).
\]

\begin{prop}\label{prop:ctkformula}
There is a canonical isomorphism
\[
\hat{C}_{t,k} \cong (T_{t,k}^{-1} (L_t \otimes_{\E^0} D_{\lkt}))^{\wedge}_{I_t}.
\]
\end{prop}
\begin{proof}
This will be proved in two steps. First, inverting $T_{t,k}$ and inverting $S_{t,k}$ in $L_{t} \otimes_{\E^0} \E^0(B\lkt^*)$ give the same ring after $I_t$-completion. Secondly, we show that inverting $T_{t,k}$ in $\E^0(B\lkt^*)$ kills the kernel of $\pi$.

It suffices to prove the first claim after taking the quotient by $I_t$. By \cite[proof of Proposition 2.5]{stapleton_tgcm}, after taking the quotient, the ring of functions applied to the quotient
\[
L_t \otimes_{\E^0} \bG_{E_n}[p^k] \rightarrow \bG_{et}[p^k]
\]
sends the coordinate $y$ of $\bG_{et}[p^k]$ to the function $x^{p^{kt}}$ on $L_t \otimes_{\E^0}\bG_{E_n}[p^k]$. Thus the set $T_{t,k}$ is the $p^{kt}$ powers of the elements in $S_{t,k}$. Inverting an element is equivalent to inverting any of its powers.

Since $a -_{\bG_{E_n}} b$ is a unit multiple of $a-b$ (for any elements $a,b$ in the maximal ideal of a complete local ring), it follows from Drinfeld's description of the kernel of $\pi$ that inverting $T_{t,k}$ in $\E^0(B\lkt^*)$ kills the kernel of $\pi$. Since $D_{\lkt}$ is a quotient of $\E^0(B\lkt^*)$, there is an isomorphism
\[
T_{t,k}^{-1}\E^0(B\lkt^*) \cong T_{t,k}^{-1}D_{\lkt}.
\]
\end{proof}

\begin{cor}\label{cor:ctk}
For all $k$, the rings $\hat{C}_{t,k}$ are Noetherian excellent normal domains. 
\end{cor}
\begin{proof}
Since $D_{\lkt}$ is a Noetherian excellent normal domain and $\hat{C}_{t,k}$ can be constructed from $D_{\lkt}$  by iterated localization and completion, \cref{prop:normal} applies. 
\end{proof}

\begin{cor}\label{cor:ct}
The ring $\colim_k \hat{C}_{t,k}$, which receives a canonical map from $C_t$, is normal.
\end{cor}
\begin{proof}
This is an immediate consequence of \Cref{cor:ctk} and Part (3) of \Cref{lem:normal}. 
\end{proof}

For more many purposes in transchromatic homotopy theory (e.g.,~\cite{bs_centralizers}), it is more convenient to work with the completion of $\colim_k \hat{C}_{t,k}$, so we end this note with the following question.

\begin{quest}
What can be said about $(\colim_k \hat{C}_{t,k})^{\wedge}_{I_t}$?
\end{quest}

\biblio
\bibliography{duality}\bibliographystyle{alpha}

\def\cprime{$'$}
\begin{thebibliography}{MGPS15}

\bibitem[BF15]{frankland}
Tobias Barthel and Martin Frankland.
\newblock Completed power operations for {M}orava {$E$}-theory.
\newblock {\em Algebr. Geom. Topol.}, 15(4):2065--2131, 2015.

\bibitem[BS15]{bs_bpetheory}
Tobias Barthel and Nathaniel Stapleton.
\newblock {B}rown--{P}eterson cohomology from {M}orava {$E$}-theory. {W}ith an
  appendix by {J}.~{H}ahn.
\newblock \url{http://arxiv.org/abs/1509.05678}, accepted for publication in
  {C}ompos.~{M}ath, 2015.

\bibitem[BS16]{bs_centralizers}
Tobias Barthel and Nathaniel Stapleton.
\newblock Centralizers in good groups are good.
\newblock {\em Algebr. Geom. Topol.}, 16(3):1453--1472, 2016.

\bibitem[Dri74]{drinfeld_ell1}
V.~G. Drinfel{\cprime}d.
\newblock Elliptic modules.
\newblock {\em Mat. Sb. (N.S.)}, 94(136):594--627, 656, 1974.

\bibitem[HKR00]{hkr}
Michael~J. Hopkins, Nicholas~J. Kuhn, and Douglas~C. Ravenel.
\newblock Generalized group characters and complex oriented cohomology
  theories.
\newblock {\em J. Amer. Math. Soc.}, 13(3):553--594 (electronic), 2000.

\bibitem[Hov95]{hovey_csc}
Mark Hovey.
\newblock Bousfield localization functors and {H}opkins' chromatic splitting
  conjecture.
\newblock In {\em The \v {C}ech centennial ({B}oston, {MA}, 1993)}, volume 181
  of {\em Contemp. Math.}, pages 225--250. Amer. Math. Soc., Providence, RI,
  1995.

\bibitem[KS16]{completionexcellent}
Kazuhiko Kurano and Kazuma Shimomoto.
\newblock Ideal-adic completion of quasi-excellent rings (after {G}abber).
\newblock \url{https://arxiv.org/pdf/1609.09246v1.pdf}, 2016.

\bibitem[MGPS15]{piperings}
Aaron Mazel-Gee, Eric Peterson, and Nathaniel Stapleton.
\newblock A relative {L}ubin-{T}ate theorem via higher formal geometry.
\newblock {\em Algebr. Geom. Topol.}, 15(4):2239--2268, 2015.

\bibitem[Sta13]{stapleton_tgcm}
Nathaniel Stapleton.
\newblock Transchromatic generalized character maps.
\newblock {\em Algebr. Geom. Topol.}, 13(1):171--203, 2013.

\bibitem[{Sta}15]{stacks-project}
The {Stacks Project Authors}.
\newblock Stacks project.
\newblock \url{http://stacks.math.columbia.edu}, 2015.

\bibitem[Str97]{strickland_finitesubgroups}
Neil~P. Strickland.
\newblock Finite subgroups of formal groups.
\newblock {\em J. Pure Appl. Algebra}, 121(2):161--208, 1997.

\end{thebibliography}
\end{document}